\documentclass[a4paper,12pt]{amsart}
\usepackage{amssymb}
\usepackage{amsmath, amsthm, amscd, amsfonts, amssymb, graphicx, color}
\usepackage{amsmath, amsthm, amscd, amsfonts, amssymb, graphicx, color}

\usepackage[cp1250]{inputenc}
\usepackage[T1]{fontenc}

\newtheorem{theorem}{Theorem}[section]
\newtheorem{lemma}[theorem]{Lemma}
\newtheorem{proposition}[theorem]{Proposition}
\newtheorem{corollary}[theorem]{Corollary}
\theoremstyle{definition}
\newtheorem{definition}[theorem]{Definition}

\theoremstyle{remark}
\newtheorem{remark}[theorem]{Remark}
\numberwithin{equation}{section}

\begin{document}
	\author{Stefan Ivkovi\'{c} }
	
	\vspace{15pt}
	
	\title{On New Approach to Fredholm theory in unital $ C^{ *} $ -algebras }

	\maketitle
	\begin{abstract}
		Motivated by Fredholm theory on the standard Hilbert module over a unital $ C^{ *} $ -algebra introduced by Mishchenko and Fomenko, we provide a new approach to axiomatic Fredholm theory in  unital $ C^{ *} $ -algebras established by Ke\v{c}ki\'{c} and Lazovi\'{c} in \cite{KL}. Our approach is equivalent to the approach introduced by Ke\v{c}ki\'{c} and Lazovi\'{c}, however, we provide new proofs which are motivated by the proofs given by Mishchenko and Fomenko in \cite{MF}.
	\end{abstract}
	
	\vspace{15pt}
	
	\begin{flushleft}
		\textbf{Keywords} Fredholm theory, Hilbert module, finite projections, K-group, index
	\end{flushleft} 
	
	\vspace{15pt}
	
	\begin{flushleft}
		\textbf{Mathematics Subject Classification (2010)} Primary MSC 47A53; Secondary MSC 46L08.
	\end{flushleft}
	
	\vspace{30pt}
	
	\section{Preliminaries}
Throughout this paper $ \mathcal{A} $ always stands for a unital $ C^{ *} $ -algebra and $B(\mathcal{A} ) $ denotes the set of all $\mathcal{A} $ - linear bounded operators on $ \mathcal{A} $ when $ \mathcal{A} $ is considered as a right Hilbert module over itself. Since $ \mathcal{A}$ is self-dual Hilbert module over itself, by \cite[Proposition 2.5.2]{MT} all operators that belong to $B(\mathcal{A} ) $ are adjointable. Moreover, by \cite[Corollary 2.5.2]{MT} the set $B(\mathcal{A} ) $ is a unital $ C^{ *} $ -algebra.\\
Let $V$ be a map from $ \mathcal{A} $ into $B(\mathcal{A} ) $ given by $ V(a) = L_{a} $ for all $ a \in \mathcal{A} $ where $L_{a} $ is the corresponding left multiplier by $a.$ Then  $V$ is an isometric *-homomorphism, and, since $ \mathcal{A} $ is unital, it follows that $V$ is in fact an isomorphism. Thus, $B(\mathcal{A} ) $ can be identified with $ \mathcal{A} $ by considering the left multipliers. \\
We recall now the following definition. \\
\begin{definition} \label{d 09}
	\cite[Definition 1.1]{KL} Let $\mathcal{A} $ be an unital $C^{*}$-algebra, and $\mathcal{F} \subseteq \mathcal{A} $ be a subalgebra which satisfies the following conditions:\\
	(i) $\mathcal{F} $ is a selfadjoint ideal in $\mathcal{A} ,$ i.e. for all $a \in \mathcal{A}, b \in \mathcal{F}  $ there holds $ab,ba \in \mathcal{F} ,$ and $a \in \mathcal{F} $ implies $a^{*} \in \mathcal{F} ;$\\
	(ii) There is an approximate unit $p_{\alpha} \in \mathcal{F}$ consisting of projections;\\
	(iii) If $p,q \in \mathcal{F} $ are projections, then there exists $v \in \mathcal{A} ,$ such that $vv^{*}=q $ and $v^{*}v \perp p, $ i.e. $v^{*}v + p $ is a projection as well.\\
	We shall call the elements of such an ideal  \textit{finite type elements}. Henceforward we shall denote this ideal by $\mathcal{F} .$
\end{definition}

Let $V$ be the isometric *-isomorphism given above. If $\mathcal{F} $  is an ideal of finite type elements in $ \mathcal{A} ,$ then it is not hard to see that $ V(\mathcal{F} ) $ is an ideal of finite type elements in $B(\mathcal{A} ),$ so we may identify $\mathcal{F} $ with $ V(\mathcal{F} ) .$

\begin{definition} \label{d 10} \cite[Definition 1.2]{KL}
	Let $\mathcal{A} $ be a unital $C^{*}-$ideal, and let $\mathcal{F} \subseteq \mathcal{A}$ be an algebra of finite type elements. In the set $ \text{ Proj}(\mathcal{F}) $ we define the equivalence relation:  
	$$p \sim q \Leftrightarrow \exists v \in \mathcal{A} \text{ } vv^{*}=p,\text{ } v^{*}v=p, $$
	i.e. Murray - von Neumann equivalence. The set $ S(\mathcal{F})=\text{Proj}(\mathcal{F})\text{ }/\sim $ is a commutative semigroup with respect to addition, and the set,$K(\mathcal{F})=G(S(\mathcal{F})),$ where $G$ denotes the Grothendic functor, is a commutative group.
\end{definition}

\begin{definition} \label{d 11} \cite[Definition 2.1]{KL}
	Let $a \in \mathcal{A} $ and $p,q $ be projections in $\mathcal{A} .$  We say that $a$ is invertible up to pair $(p, q)$ if there exists some $b \in \mathcal{A} $  such that 
	$$(1-q)a(1-p)b=1-q, \text{  } b(1-q)a(1-p)=1-p. $$
	We refer to such $b$ as almost inverse of $a,$ or $(p,q)$-inverse of $a.$
\end{definition}
Next, we recall also the following definition regarding Hilbert modules.
\begin{definition} \label{DP D04} 
	\cite[Definition 2.3.1]{MT} A closed submodule $\mathcal{N} $ in a Hilbert $C^{*}$-module $ \mathcal{M}$ is called (topologically) complementable if there exists a closed submodule $\mathcal{L} $ in $\mathcal{M} $ such that  $\mathcal{N}+\mathcal{L}=\mathcal{M} ,\mathcal{N} \cap \mathcal{L}=0.$
\end{definition}	
By the symbol $\tilde{ \oplus} $ \index{$\tilde{ \oplus} $} we denote the direct sum of modules as given in \cite{MT}.

Thus, if $M$ is a Hilbert $C^{*}$-module and $M_{1}, M_{2}$ are two closed submodules of $M,$ we write $M=M_{1} \tilde \oplus M_{2}$ if $M_{1} \cap M_{2}=\lbrace 0 \rbrace$ and $M_{1}+ M_{2}=M.$ If, in addition $M_{1}$  and $M_{2}$ are mutually orthogonal, then we write $M=M_{1} \oplus M_{2}.$
 \begin{remark} \label{r10 r01}
	If $\sqcap \in B(\mathcal{A}) $ is a (skew ) projection, then, since $Im \sqcap $ is closed, by \cite[Theorem 2.3.3]{MT} we get that $Im \sqcap $ is  complementable. Hence, every closed and complementable submodule $M$ of $\mathcal{A} $ is orthogonally complementable. The corresponding orthogonal projection onto $M$ will be denoted by $ P_{M } .$
\end{remark}

\begin{remark} \label{skewprojfinite}
	If $ \sqcap \in B(\mathcal{A})$ has closed range and $ P_{Im \sqcap} \in \mathcal{F} ,$ then, since $P_{Im \sqcap} \sqcap = \sqcap$ and $\mathcal{F}$ is an ideal, we get that $ \sqcap \in \mathcal{F} .$ 
\end{remark}
At the end of this section we give also the following technical results.
\begin{lemma} \label{polar} 
	Let $ T \in B(\mathcal{A})  $ and suppose that $ Im T $ is closed. Then $ Im (T^{*} T)^{1/2 } $ is closed.
	
\end{lemma}
\begin{proof}
	By the proof of \cite[Theorem 2.3.3]{MT} we have that $ Im T^{*} $ is closed when $ ImT $ is closed. In addition, $ ImT $ is orthogonally complementable in $ \mathcal{A}$ by \cite[Theorem 2.3.3]{MT} . Let $P$ denote the orthogonal projection onto $ ImT .$ Then $ T=PT, $ hence $ T^{*} = T^{*} P. $ It follows that  $ Im T^{*} = Im T^{*} P =Im T^{*} T  ,$ so $ Im T^{*} T $ is closed. Hence, $$ \mathcal{A} = Im T^{*} T  \oplus \ker T^{*} T   $$ by \cite[Theorem 2.3.3]{MT}, and $T^{*} T $ maps $Im T^{*} T $ isomorphically onto itself, which gives that $ ((T^{*} T )^{ 1/2 })_{ \vert Im T^{*} T } $ is bounded below.\\
	Next, it is obvious that $  \ker T^{*} T = \ker (T^{*} T)^{1/2 } . $ Indeed, if $ x \in \ker T^{*} T ,$ then $ \langle (T^{*} T)^{1/2 } x,(T^{*} T)^{1/2 } x \rangle = \langle (T^{*} T) x, x \rangle  = 0 ,$ so $ \ker T^{*} T \subseteq \ker (T^{*} T)^{1/2 } ,$ whereas the opposite inclusion is obvious. Thus we obtain that   $$ \mathcal{A} = Im T^{*} T  \oplus \ker (T^{*} T)^{ 1/2 }   ,$$ so $ (T^{*} T)^{1/2 } ( ImT^{*} T ) = Im (T^{*} T)^{1/2 } .$ However, since $ ((T^{*} T )^{ 1/2 })_{ \vert Im T^{*} T } $ is bounded below, we must have that $ Im (T^{*} T)^{1/2 } $ is closed. 
\end{proof} 
\begin{corollary} \label{cpolar}
Let $ T \in B(\mathcal{A})  $ and suppose that $ Im T $ is closed. Then $T$ admits polar decomposition.
\end{corollary}
\begin{proof}
	By Lemma \ref{polar} we have that $ Im (T^{*} T)^{1/2 } $ is closed. Hence, by \cite[Theorem 2.3.3]{MT} we get that $ImT$ and $ Im (T^{*} T)^{1/2 } $  are orthogonally complementable, so $T$ admits polar decomposition. 
\end{proof}

\section{Main results} 

We start with the following lemma. 

\begin{lemma} \label{r10 l01}
	Let $ \tilde{P}, \tilde{Q} \in Proj (\mathcal{A}) .$ Then $\tilde{P} \sim \tilde{Q}$ if and only if $Im \tilde{P} \cong Im \tilde{Q} .$
\end{lemma}
  
\begin{proof}
	Suppose that $Im \tilde{P} \cong Im \tilde{Q} $ and let $\mathcal{U} $ be an isomorphism from $Im \tilde{P} $ onto $Im \tilde{Q} $.
	Set $T:= J \mathcal{U} \tilde{P}  $ where $J: Im \tilde{Q} \rightarrow \mathcal{A} $ is inclusion. Then $T \in  (\mathcal{A}) $  and $Im T = Im \tilde{Q} $ is closed. Hence, by Corollary \ref{cpolar} we deduce that  $T$ admits polar decomposition. The partial isometry $V$ from this decomposition satisfies that $V^{*}V = P_{\ker T^{\perp}} = \tilde{P} $, and $VV^{*} = P_{Im T} = \tilde{Q} .$\\
	 Conversely, if $\tilde{P} \sim \tilde{Q} ,$ then there exists some $ V \in \mathcal{A}$ such that $VV^{*} = \tilde{Q}   $ and $V^{*}V = \tilde{P} .$ Then $\tilde{Q} V \tilde{P} $ is the desired isomorphism.
\end{proof}

\begin{proposition} \label{r10 pr01}
	Let $\lbrace P_{\alpha}\rbrace_{\alpha} $ be an approximate unit for $\mathcal{F} $ consisting of orthogonal projections and $N$ be a closed, complementable submodule of $\mathcal{A} $ such that $P_{N} \in \mathcal{F} .$ Then there exists some $\alpha_{0} $ and a closed submodule $M$ of $\mathcal{A} $ such that $Im(I-P_{\alpha_{0}}) \subseteq M $ and $ \mathcal{A} = M \tilde{\oplus} N .$
\end{proposition}

\begin{proof}
	Choose $\alpha_{0} $ sufficiently large such that $\parallel P_{N}-P_{\alpha_{0}}P_{N} \parallel < 1 .$ Then we get that $\parallel P_{N}-P_{N}P_{\alpha_{0}}P_{N} \parallel < 1 $ which gives that $P_{N}P_{\alpha_{0}}P_{N} $ is invertible in the corner $ C^{*} $ - algebra $P_{N} \mathcal{A} P_{N} .$ It is not hard to deduce then that ${P_{\alpha_{0}}}_{\vert_{N}}  $ must be bounded below. So $Im P_{\alpha_{0}} P_{N}  $ is closed, thus orthogonally complementable by  \cite[Theorem 2.3.3]{MT}.\\
	  Let $M= (Im P_{\alpha_{0}} P_{N})^{\perp} . $ Then $Im (I- P_{\alpha_{0}}) \subseteq M .$\\
	   Set $\tilde{P} $ to be the orthogonal projection onto $Im  P_{\alpha_{0}} P_{N}  .$ Then $ \tilde{P} \leq  P_{\alpha_{0}} $ and therefore  $\tilde{P}_{\vert_{N}} = \tilde{P} {P_{\alpha_{0}}}_{\vert_{N}} = {P_{\alpha_{0}}}_{\vert_{N}} .$ 
	Since ${P_{\alpha_{0}}}_{\vert_{N}} $ is an isomorphism onto $Im P_{\alpha_{0}}P_{N}  ,$ it follows that $\tilde{P}_{\vert_{N}} $ an isomorphism onto $Im P_{\alpha_{0}}P_{N} = Im \tilde{P} .$ It is then not hard to deduce that $\mathcal{A} = M \tilde{\oplus} N .$ 
\end{proof}

\begin{lemma} \label{r10 l03}
	 Let $N $ be a closed complementable submodule of $\mathcal{A} $ such that $P_{N} \in \mathcal{F} .$ Suppose that $F \in B (\mathcal{A})  $ is such that $F_{\vert_{N}}$ an isomorphism. Then $F(N) $ is  complementable and $P_{F(N)} \in \mathcal{F} .$
\end{lemma}

\begin{proof}
	Since $N$ is complementable, it is  orthogonally complementable by Remark \ref{r10 r01}. Now, $ F(N)=ImFP_{N},$ so by \cite[Theorem 2.3.3]{MT}, $ F(N)$  is  complementable in $\mathcal{A} .$ Since $ F_{\vert_{N}}$ is an isomorphism onto $F(N) ,$ we have that $ P_{N} \sim P_{F(N)}$ by Lemma \ref{r10 l01}. 
\end{proof}

\begin{lemma} \label{r10 l04}
	Let $M,N$ be two  closed, complementable submodules of $\mathcal{A} $ such that $P_{N}, P_{M} \in \mathcal{F} .$ Suppose that $M \cap N = \lbrace 0 \rbrace $ and that $M+N $ is closed and complementable in $\mathcal{A} .$ Then $P_{M \tilde{\oplus} N} \in \mathcal{F} $ and $[P_{M \tilde{\oplus} N}] = [P_{M}] + [P_{N}] .$
\end{lemma} 
 
 \begin{proof}
 	Note first that by Remark \ref{r10 r01} the submodules $M, N$ and $M \tilde{\oplus} N $ are orthogonally complementable in $ \mathcal{A} .$ Since  $M$ is orthogonally complementable and $M \subseteq M \tilde{\oplus} N ,$ by \cite[Lemma 2.6]{IS3} we have that $M$ is orthogonally complementable in $ M \tilde{\oplus} N  .$ Let $R $ be the orthogonal complement of $M$ in $M \tilde{\oplus} N .$ Then, since $M \tilde{\oplus} N= M \oplus R  ,$ it is clear that $R \cong N .$ By Lemma \ref{r10 l01} $P_{R} \sim P_{N} \in \mathcal{F} .$ Indeed, 
 	since $ M \tilde{\oplus} N $ is orthogonally complementable in $\mathcal{A} $, then $R $ is orthogonally complementable in $\mathcal{A} .$ Now, we have $P_{M \tilde{\oplus} N} = P_{M}+P_{R} \in \mathcal{F} .$ Moreover, $[ P_{M \tilde{\oplus} N} ]= [P_{M}]+[P_{R}] = [P_{M}]+[P_{N}] ,$ as $P_{R} \sim P_{N} .$ 
 \end{proof}

\begin{lemma} \label{r10 l00}
	Let $F \in B(\mathcal{A}) $ and suppose that 
	$$ \mathcal{A} = M_{1} \tilde{\oplus} N_{1}  \stackrel{F}{\longrightarrow} M_{2} \tilde{\oplus} N_{2} = \mathcal{A} $$ 
	is a decomposition with respect to which $F$
	has the matrix 
	$ 
	\begin{pmatrix}
		F_{1}  & 0 \\
		0 & F_{4}  
	\end{pmatrix} 
	$  
	where $F_{1} $ is an isomorphism. Then, with respect to the decomposition 
	$$ \mathcal{A} = N_{1}^{\perp} \oplus  N_{1}  \stackrel{F}{\longrightarrow} F(N_{1}^{\perp}) \tilde{\oplus} N_{2} = \mathcal{A} ,$$ 
	$F $ has the matrix 
		$ 
	\begin{pmatrix}
	\tilde	F_{1}  & 0 \\
		0 & F_{4}  
	\end{pmatrix} 
	$  
	where $ \tilde{F_{1}}$ is an isomorphism. 
\end{lemma}

\begin{proof}
	Let $ \mathcal{A} = M_{1} \tilde \oplus N_{1} \stackrel{F}{\longrightarrow} M_{2} \tilde \oplus N_{2}= \mathcal{A}$ be a  decomposition with respect to which  $F$ has the matrix
		$ 
	\begin{pmatrix}
		F_{1}  & 0 \\
		0 & F_{4}  
	\end{pmatrix} 
	$  
	where $F_{1} $ is an isomorphism.Observe first that, since $N_{1} $ is orthogonally complementable by Remark \ref{r10 r01}, then 
	$$\mathcal{A} = M_{1}   \tilde \oplus N_{1} = N_{1} \oplus N_{1}^\bot ,$$ 
	so $\sqcap_{{M_{1}}_{{\mid}_{N_{1}^\bot}}} $ is an isomorphism from $N_{1}^\bot$ onto $M_{1},$ where $\sqcap_{{M_{1}}_{{\mid}_{N_{1}^\bot}}} $ stands for the projection onto $M_{1}$ along $N_{1}$ restricted to $N_{1}^\bot$. Observe next that, since $F(M_{1})=M_{2} $ and $F(N_{1})\subseteq N_{2}$, we have  ${\sqcap_{{M_{2}}}   F_{{\mid}_{N_{1}^\bot}}}=F \sqcap_{{M_{1}}_{{\mid}_{N_{1}^\bot}}},$ where $\sqcap_{{M_{2}}} $ stands for the projection onto $M_{2} $ along $N_{2}.$ Since $F_{{\mid}_{M_{1}}}$ is an isomorphism, it follows that $\sqcap_{{M_{2}}} F_{{\mid}_{N_{1}^\bot}}=F \sqcap_{{M_{1}}_{{\mid}_{N_{1}^\bot}}} $ is an isomorphism as a composition of isomorphisms. Hence, with respect to  the decomposition
	$$\mathcal{A} = {N_{1}^\bot} \oplus {N_{1}} {\stackrel{F}{\longrightarrow}}  {M_{2}} \tilde \oplus N_{2}=\mathcal{A},$$
	$F$ has the matrix
$ 
\begin{pmatrix}
	 \tilde{F_{1}}  & 0 \\
	\tilde{F_{3
	}} & F_{4}  
\end{pmatrix} 
$  
	where $\tilde F_{1}=\sqcap_{{M_{2}}} F_{{\mid}_{N_{1}^\bot}} $is an isomorphism. Using the technique of diagonalization as in the proof of \cite[Lemma 2.7.10]{MT}, we deduce that there exists an isomorphism $V$ such that 
	$$ \mathcal{A} = N_{1}^{\perp} \oplus N_{1} \stackrel{F}{\longrightarrow} V(M_{2}) \tilde \oplus V(N_{2})= \mathcal{A} $$
	is a decomposition with respect to which $F$ has the matrix  $ 
	\begin{pmatrix}
	\tilde{ \tilde{F_{1}}}  & 0 \\
		0 & F_{4}  
	\end{pmatrix} 
	$  
	where $ \tilde{ \tilde{F_{1}}}$ is an isomorphism. Moreover, by the construction of $V$ we have $V(N_{2})=N_{2}.$ Hence $$ \mathcal{A} = F ( N_{1}^{\perp}) \tilde \oplus N_{2}.$$ 
	
	Thus, with respect to the decomposition 
	$$ \mathcal{A} = N_{1}^{\perp} \oplus N_{1}  \stackrel{F}{\longrightarrow} F(N_{1}^{\perp}) \tilde{\oplus} N_{2} = \mathcal{A} ,$$ $F$ has the desired matrix.
\end{proof}
In exactly the same way we can prove the following corollary.
 
\begin{corollary} \label{extracor7}
Let $F \in B(\mathcal{A}) $ and suppose that 
$$ \mathcal{A} = M_{1} \tilde{\oplus} N_{1}  \stackrel{F}{\longrightarrow} M_{2} \tilde{\oplus} N_{2} = \mathcal{A} $$ 
is a decomposition with respect to which $F$
has the matrix 
$ 
\begin{pmatrix}
	F_{1}  & 0 \\
	0 & F_{4}  
\end{pmatrix} 
$  
where $F_{1} $ is an isomorphism. If there exists a closed submodule $ \tilde M $ of $ \mathcal{A} $ such that $ \mathcal{A} = \tilde M \tilde \oplus N_{1} ,$ then $ F $ has the matrix
$ 
\begin{pmatrix}
\tilde	F_{1}  & 0 \\
	0 & F_{4}  
\end{pmatrix} 
$ with respect to the decomposition 
$$ \mathcal{A} = \tilde M \tilde{\oplus} N_{1}  \stackrel{F}{\longrightarrow} F( \tilde M ) \tilde{\oplus} N_{2} = \mathcal{A} $$ 
where $ \tilde{F_{1}}$ is an isomorphism. 
\end{corollary}
Thanks to Lemma \ref{r10 l00} we obtain the following useful characterization of invertibility up to pair of orthogonal projections.
\begin{lemma} \label{r10 l02}
	Let $ F \in \mathcal{A} .$ Then $F$ has the matrix 
	$ 
	\begin{pmatrix}
		F_{1}  & 0 \\
		0 & F_{4}  
	\end{pmatrix} 
	$  
	with respect to the decomposition 
	$$\mathcal{A} = M_{1} \tilde{\oplus} N_{1}  \stackrel{F}{\longrightarrow} M_{2} \tilde{\oplus} N_{2} = \mathcal{A}   $$ 
	where $F_{1} $ is an isomorphism if and only if $F$ is invertible up to $(P,Q)$ where $P \sim P_{N_{1}} $ and $Q \sim P_{N_{2}}.$
\end{lemma}

\begin{proof}
 By the proof of Lemma \ref{r10 l00}, if $F$ has a decomposition  
$$ \mathcal{A} = M_{1} \tilde{\oplus} N_{1}  \stackrel{F}{\longrightarrow} M_{2} \tilde{\oplus} N_{2} = \mathcal{A} ,$$ with respect to which $F$ has the matrix 
then $ 
\begin{pmatrix}
	F_{1}  & 0 \\
	0 & F_{4}  
\end{pmatrix} 
$ where $ 	F_{1} $ is an isomorphism, then 
$F$ is invertible up to pair ( $ P_{N_{1}},  P_{ F( N_{1}^{\perp} )^{\perp}} $). However, we have 
$$  \mathcal{A} = F( N_{1}^{\perp}) \oplus F( N_{1}^{\perp} )^{\perp}  = F( N_{1}^{\perp}) \tilde{\oplus} N_{2}, $$ 
hence $N_{2} \cong F( N_{1}^{\perp} )^{\perp} .$ By Lemma \ref{r10 l01} , $P_{N_{2}} \sim  P_{ F( N_{1}^{\perp} )^{\perp}} .$ \\
Conversely, if $F$ is invertible up to pair of orthogonal  projections  $(P,Q),$ then by the proof of  \cite[Lemma 2.7.10]{MT}, $F$ has decomposition 
$$ \mathcal{A} = M_{1} \tilde{\oplus} N_{1}  \stackrel{F}{\longrightarrow} M_{2} \tilde{\oplus} N_{2} = \mathcal{A} $$ 
where $N_{1} \cong Im P $ and $N_{2} \cong Im Q .$ By Lemma \ref{r10 l01} we have that $P_{N_{1}} \sim P $ and $ P_{N_{2}} \sim Q .$
\end{proof} 
We introduce now the following definition. 
 \begin{definition}
 	Let $F \in B(\mathcal{A}) . $  We say that $ F \in \mathcal{M} \mathcal{K} \Phi (\mathcal{A} ) $ if there exists a decomposition 
 	$$ \mathcal{A} = M_{1} \tilde{\oplus} N_{1}  \stackrel{F}{\longrightarrow} M_{2} \tilde{\oplus} N_{2} = \mathcal{A} $$ 
 	with respect to which $F$
 	has the matrix 
 	$ 
 	\begin{pmatrix}
 		F_{1}  & 0 \\
 		0 & F_{4}  
 	\end{pmatrix} 
 	$  
 	where $F_{1} $ is an isomorphism and $ P_{N_{1} }, P_{N_{2} } \in \mathcal{F} .$ We put then 
 	$$ index F = [P_{N_{1} } ] - [ P_{N_{2} } ]$$ in $ K(\mathcal{F} ) .$
 \end{definition}
Notice that since $ N_{1}  $ and $ N_{2}  $ are closed and complementable, by Remark \ref{r10 r01} they are orthogonally complementable, hence $P_{N_{1} } $ and $P_{N_{2} }$ are well defined. It remains to prove that the index is well defined. 
\begin{theorem} \label{indextheorem}
	The index is well defined.
\end{theorem}
\begin{proof}
	Let 
	$$ \mathcal{A} = M_{1} \tilde{\oplus} N_{1}  \stackrel{F}{\longrightarrow} M_{2} \tilde{\oplus} N_{2} = \mathcal{A} $$  
	$$ \mathcal{A} = M_{1}^{\prime} \tilde{\oplus} N_{1}^{\prime}  \stackrel{F}{\longrightarrow} M_{2}^{\prime} \tilde{\oplus} N_{2}^{\prime} = \mathcal{A} $$

	be two $ \mathcal{M} \mathcal{K} \Phi$ -decompositions for $F.$ By Proposition \ref{r10 pr01} there exist closed submodules $\tilde{M_{1}}, \tilde{M_{1}}^{\prime} $ of $\mathcal{A} $ such that 
	$ \mathcal{A} = \tilde{M_{1}} \tilde{\oplus} N_{1}  = \tilde{M_{1}}^{\prime} \tilde{\oplus} N_{1}^{\prime}$	
	and $Im (I-P_{\alpha}) \subseteq  \tilde{M_{1}} \cap  \tilde{M_{1}}^{\prime} $ for sufficiently large $\alpha .$ 
	By Corollary \ref{extracor7} the operator $F$ has the matrices 
		$ 
	\begin{pmatrix}
		F_{1}  & 0 \\
		0 & F_{4}  
	\end{pmatrix} 
	,$  
	$ 
	\begin{pmatrix}
		F_{1}  & 0 \\
		0 & F_{4}  
	\end{pmatrix} 
	,$ 
	with respect to the decompositions 
	$$ \tilde{M_{1}} \tilde{\oplus} N_{1}  \stackrel{F}{\longrightarrow} F(\tilde{M_{1}}) \tilde{\oplus} N_{2} ,$$ 
	$$\tilde{M_{1}}^{\prime} \tilde{\oplus} N_{1}^{\prime}  \stackrel{F}{\longrightarrow} F(\tilde{M_{1}}^{\prime}) \tilde{\oplus} N_{2}^{\prime} ,$$
	respectively, where $F_{1} $ and $F_{1}^{\prime} $ are isomorphisms.\\
	Now, since $ Im(I-P_{\alpha}) \subseteq \tilde{M_{1}} \cap \tilde{M_{1}}^{\prime},$ by  \cite[Lemma 2.6]{IS3} there exist closed submodules $\mathcal{R}$ and $\mathcal{R}^{\prime} $ of $\mathcal{A}$ such that 
	$ \tilde{M_{1}} =  Im(I-P_{\alpha}) \oplus \mathcal{R} $ and $ \tilde{M_{1}}^{\prime} =  Im(I-P_{\alpha}) \oplus \mathcal{R}^{\prime} .$ As in the proof of \cite[Lemma 2.7.13]{MT} we obtain new $ \mathcal{M} \mathcal{K} \Phi$ - decompositions  
	$$ \mathcal{A}= Im(I-P_{\alpha}) \tilde{\oplus} \mathcal{R} \tilde \oplus N_{1} \stackrel{F}{\longrightarrow} F(Im(I-P_{\alpha})) \tilde{\oplus} F( \mathcal{R} )\tilde \oplus N_{2} = \mathcal{A} ,$$
	$$ \mathcal{A} =  Im(I-P_{\alpha}) \tilde{\oplus} \mathcal{R}^{\prime}  \tilde{\oplus} N_{1}^{\prime}  \stackrel{F}{\longrightarrow}  F(Im(I-P_{\alpha})) \tilde{\oplus} F(\mathcal{R}^{\prime} ) \tilde{\oplus} N_{2}^{\prime} = \mathcal{A} $$
	for the operator $ F.$
	Indeed, since 
	$$ \mathcal{A} =\tilde{M_{1}} \tilde{\oplus} N_{1}  = Im(I-P_{\alpha}) \tilde{\oplus} \mathcal{R} \tilde{\oplus} N_{1} =Im(I-P_{\alpha})  \oplus Im \text{ }P_{\alpha}, $$ 
	we have that $ \mathcal{R} \tilde{\oplus} N_{1} \cong Im \text{ }P_{\alpha}.$ Hence $P_{\mathcal{R} \tilde{\oplus} N_{1}} \sim P_{\alpha}$ by Lemma \ref{r10 l01}, so $P_{\mathcal{R} \tilde{\oplus} N_{1}} \in \mathcal{F} .$ Since $P_{\mathcal{R}} \leq  P_{\mathcal{R} \tilde{\oplus} N_{1}} ,$ it follows that $ P_{\mathcal{R}} \in \mathcal{F} $ (as 
	$ P_\mathcal{R} =P_\mathcal{R} P_{\mathcal{R} \tilde{\oplus} N_{1}} $ and $\mathcal{F} $ is an ideal). 
	Similarly,  $P_{\mathcal{R}^{\prime} \tilde{\oplus} N_{1}^{\prime}} \sim P_{\alpha} $ and thus  $P_{\mathcal{R}^{\prime} },P_{\mathcal{R}^{\prime} \tilde{\oplus} N_{1}^{\prime}} \in \mathcal{F} .$ Then, by Lemma \ref{r10 l03}, as $F_{\vert_{\mathcal{R}}} $ and $F_{\vert_{\mathcal{R^{\prime}}}} $ are isomorphisms, we get that $P_{F({\mathcal{R})}}, P_{F({\mathcal{R^{\prime}})}} \in \mathcal{F} .$ Hence, by Lemma \ref{r10 l04} we obtain that $P_{ F( \mathcal{R} )\tilde{\oplus} N_{2} }  ,$ $P_{ F( \mathcal{R} )\tilde{\oplus} N_{2}^{\prime}}  \in \mathcal{F}$   
	and 
	$ [P_{F(\mathcal{R}) \tilde{\oplus} N_{2}}] = [P_{F(\mathcal{R})}] +[P_{N_{2}}] ,$
	 $[P_{F{(\mathcal{R}^{\prime})} \tilde{\oplus} N_{2}^{\prime}}] = [P_{ F (\mathcal{R^{\prime}} )}] + [P_{N_{2}^{\prime}}].$ \\
	Next, since $P_{\mathcal{R}} \sim P_{F(\mathcal{R})}$ and $P_{\mathcal{R}^{\prime}} \sim P_{F(\mathcal{R}^{\prime})}$ 
	we have that $[P_{F(\mathcal{R})}] = [P_{\mathcal{R}}] $ and $ [ P_{F(\mathcal{R}^{\prime})} ] = [P_{\mathcal{R}^{\prime}}] .$ By Lemma \ref{r10 l04} we also have 
	$[P_{\alpha}] = [P_{\mathcal{R} \tilde{\oplus} N_{1}}] = [P_{\mathcal{R}}] + [P_{N_{1}}]$ 
	and 
	$[P_{\alpha}] = [P_{ \mathcal{R}^{\prime} \tilde{\oplus} N_{1}^{\prime} }] = [P_{ \mathcal{R}^{\prime} }] + [P_{ N_{1}^{\prime} }] .$ 
 	
	On the other hand, since 
	$$\mathcal{A} = F ( Im \text{ } P_{\alpha} ) \tilde{\oplus} F(\mathcal{R}) \tilde{\oplus} N_{2} = F (  Im \text{ } P_{\alpha} )  \tilde{\oplus} F( \mathcal{R}^{\prime}) \tilde{\oplus}  N_{2}^{\prime} ,$$ 
	we have that 
	$ F( \mathcal{R}) \tilde{\oplus}  N_{2} \cong  F( \mathcal{R}^{\prime}) \tilde{\oplus}  N_{2}^{\prime}   ,$ 
	hence, by Lemma \ref{r10 l01} and Lemma \ref{r10 l04} we get that 
	$[ P_{F (R)} ] + [ P_{N_{2}} ] = [ P_{F (R^{\prime})} ] + [ P_{N_{2}^{\prime}} ] .$ 
		
	Putting all this together, we obtain that 
	$$[ P_{ \mathcal{R} \tilde{\oplus} N_{1} } ] - [ P_{ F( \mathcal{R} ) \tilde{\oplus} N_{2} } ] =
	[ P_{ \mathcal{R}^{\prime} \tilde{\oplus} N_{1}^{\prime} } ] - [ P_{ F( \mathcal{R}^{\prime} ) \tilde{\oplus} N_{2}^{\prime} } ]   ,$$ 
	however, 
	$$[ P_{ \mathcal{R} \tilde{\oplus} N_{1} } ] - [ P_{ F( \mathcal{R} ) \tilde{\oplus} N_{2} } ] = [P_{\mathcal{R}}] + [P_{N_{1}}] - [ P_{ F( \mathcal{R} )}] - [P_{N_{2}}] = [P_{N_{1}}] -  [P_{N_{2}}]$$
	and similarly 
	$$[ P_{ \mathcal{R}^{\prime} \tilde{\oplus} N_{1}^{\prime} } ] -  [ P_{ F( \mathcal{R}^{\prime} ) \tilde{\oplus} N_{2}^{\prime}}] =  [P_{N_{1}^{\prime}}] -  [P_{N_{2}^{\prime}}] .$$  
\end{proof} 	
	
	Thanks to Lemma \ref{r10 l01}, we can prove the next result in a similar way as  \cite[Lemma 2.7.10]{MT} . For the convenience of readers, we give the full proof here. 
	
	\begin{proposition} \label{indexproposition}
		Let $ F,D \in \mathcal{M} \mathcal{K} \Phi ( \mathcal{A} ). $ Then $ DF \in \mathcal{M} \mathcal{K} \Phi ( \mathcal{A} ) $ and 
		$$ \text {\rm index } DF=\text {\rm index } D+\text {\rm index } F.$$  
	\end{proposition}
	\begin{proof}  Let 
	$$ \mathcal{A} = M_{1} \tilde{\oplus} N_{1}  \stackrel{F}{\longrightarrow} M_{2} \tilde{\oplus} N_{2} = \mathcal{A} $$
	be an $\mathcal{M} \mathcal{K} \Phi$-decomposition for $F.$ By Proposition \ref{r10 pr01} there exists some $\alpha_{0} $ and a closed submodule $\tilde{M} $ such that $Im (I-P_{\alpha_{0}}) \subseteq \tilde{M} $ and $ \mathcal{A} = \tilde{M} \tilde{\oplus} N_{2}.$ If $ \sqcap $ denotes the projection onto $ \tilde{M} $ along $ N_{2} ,$ then $\sqcap_{\vert_{M_{2}}} $ is an isomorphism onto $\tilde{M} .$ Let $V$ be the operator with the matrix 
	$ 
	\begin{pmatrix}
		\sqcap  & 0 \\
		0 & 1  
	\end{pmatrix} 
	$  
	with respect to the decomposition 
	$$ \mathcal{A}=  M_{2} \tilde{\oplus} N_{2}  \stackrel{V}{\longrightarrow} \tilde{M} \tilde{\oplus} N_{2} = \mathcal{A} .$$ 
	Then $V$ is an isomorphism on $\mathcal{A} ,$ and with respect to the decomposition 
	$$ \mathcal{A} = M_{1} \tilde{\oplus} N_{1}  {\longrightarrow} \tilde{M} \tilde{\oplus} N_{2} = \mathcal{A} ,$$ 
	the operator $VF$ has the matrix 
	$ 
	\begin{pmatrix}
		(VF)_{1}  & 0 \\
		0 & (VF)_{4} 
	\end{pmatrix} 
	$  
	where $(VF)_{1} $ is an isomorphism. Hence, $ \text {\rm index } VF=\text {\rm index } F .$ \\
	Note that if 
	$$ \mathcal{A} = M_{1}^{\prime} \tilde{\oplus} N_{1}^{\prime}  \stackrel{D}{\longrightarrow} M_{2}^{\prime} \tilde{\oplus} N_{2}^{\prime} = \mathcal{A} $$ 
	is an $\mathcal{M}  \mathcal{K} \Phi$-decomposition for $D,$ then 
	$$ \mathcal{A} = V ( M_{1}^{\prime} ) \tilde{\oplus} V ( N_{1}^{\prime} )  {\longrightarrow} M_{2}^{\prime} \tilde{\oplus} N_{2}^{\prime} = \mathcal{A} $$  
	an $\mathcal{M}  \mathcal{K} \Phi $-decomposition for $DV^{-1} $ and $index  DV^{-1} = index D .$ This follows from Lemma \ref{r10 l03} since $ V ( N_{1}^{\prime} ) \cong N_{1}^{\prime}  ,$ hence $ P_{ V ( N_{1}^{\prime} ) } \sim P_{ N_{1}^{\prime} } \in \mathcal{F}.$ Now, since $P_{ V ( N_{1}^{\prime} ) } \in \mathcal{F} ,$ by Proposition \ref{r10 pr01} we can find some $ \alpha_{1} \geq \alpha_{0}$ and a closed
	submodule $ \tilde M^{\prime} $ such that $\mathcal{A} = \tilde M^{\prime}  \tilde{\oplus} V ( N_{1}^{\prime} )  $ and $ Im (I-P_{\alpha_{1}}) \subseteq \tilde  M^{\prime}.$ Then, by Corollary \ref{extracor7}, the decomposition $$\mathcal{A} = \tilde M^{\prime}  \tilde{\oplus} V ( N_{1}^{\prime} ) {\longrightarrow} DV^{-1} (\tilde M^{\prime}) \tilde{\oplus}  N_{2}^{\prime} = \mathcal{A}  $$ is also an $\mathcal{M}  \mathcal{K} \Phi $-decomposition for $DV^{-1} .$ Moreover, $ Im (I-P_{\alpha_{1}}) \subseteq \tilde M \cap \tilde M^{\prime} .$ By \cite[Lemma 2.6]{IS3} there exist closed submodules $R, R^{ \prime} \subseteq \mathcal{A} $ such that $$ \tilde M = Im (I-P_{\alpha_{1}}) \oplus R , \tilde M^{ \prime} = Im (I-P_{\alpha_{1}}) \oplus R^{ \prime} .$$

	As in  the first part of the proof of Theorem \ref{indextheorem}, we obtain $\mathcal{M}  \mathcal{K} \Phi $-decompositions  $$\mathcal{A} = (VF)_{1} ^{ -1} (Im (I-P_{\alpha_{1}})) \tilde{\oplus} ((VF)_{1} ^{ -1} (R) \tilde{\oplus}  N_{1} ) \stackrel{VF}{\longrightarrow}  Im (I-P_{\alpha_{1}}) \tilde{\oplus} (R\tilde{\oplus} N_{2}) = \mathcal{A} ,$$ $$\mathcal{A} = Im (I-P_{\alpha_{1}}) \tilde{\oplus} ( R ^{\prime} \tilde{\oplus} V (N_{1} ^{\prime} ) ) \stackrel{DV^{-1}}{\longrightarrow} DV^{-1} (Im (I-P_{\alpha_{1}}))  \tilde{\oplus} (DV^{-1} (R ^{\prime})\tilde{\oplus}  N_{2}^{\prime}   ) =  \mathcal{A} $$ for the operators $VF$ and $DV^{-1} ,$ respectively, where $ (VF)_{1} ^{ -1} (R) \cong R , \text{ } R ^{\prime} \cong DV^{-1} (R ^{\prime}) .$  Finally, since $$\mathcal{A} =Im (I-P_{\alpha_{1}}) \tilde{\oplus} R \tilde{\oplus}  N_{2} = Im (I-P_{\alpha_{1}}) \tilde{\oplus} R ^{\prime} \tilde{\oplus} V (N_{1} ^{\prime} )  ,$$ we get that $ R \tilde{\oplus}  N_{2} \cong R ^{\prime} \tilde{\oplus} V (N_{1} ^{\prime} ) .$ By Lemma \ref{r10 l01} and Lemma \ref{r10 l04} we deduce that index $D +$ index $ F$=index $ DV^{-1} +$ index $VF = [P_{ R ^{\prime} \tilde{\oplus} V (N_{1} ^{\prime} )}] - [ P_{DV^{-1} (R ^{\prime})\tilde{\oplus}  N_{2}^{\prime}  } ] + [ P_{VF_{1} ^{ -1} (R) \tilde{\oplus}  N_{1} } ] - [P_ { R\tilde{\oplus} N_{2}}] = [ P_{VF_{1} ^{ -1} (R) \tilde{\oplus}  N_{1} } ]  - [ P_{DV^{-1} (R ^{\prime})\tilde{\oplus}  N_{2}^{\prime}  } ] .$\\
	 On the other hand, it is clear that the operator $ DF$ has the matrix $ 
	 \begin{pmatrix}
	 	(DF)_{1}  & (DF)_{2} \\
	 	0 & (DF)_{4} 
	 \end{pmatrix} 
	 $  with respect to the decomposition $$\mathcal{A} = (VF)_{1} ^{ -1} (Im (I-P_{\alpha_{1}})) \tilde{\oplus} ((VF)_{1} ^{ -1} (R) \tilde{\oplus}  N_{1} ) \stackrel{DF}{\longrightarrow} DV^{-1} (Im (I-P_{\alpha_{1}}))  \tilde{\oplus} (DV^{-1} (R ^{\prime})\tilde{\oplus}  N_{2}^{\prime}   ) =  \mathcal{A} ,$$ where $ (DF)_{1} $ is an isomorphism (because $ DF = DV^{-1} VF $).  Hence, as in the proof of \cite[Lemma 2.7.10]{MT}  we can find a unitary operator $ U $ such that $ DF $ has the matrix $\begin{pmatrix}
	 	(DF)_{1}  & 0 \\
	 	0 & \tilde{(DF)_{4} }
	 \end{pmatrix} $
	with respect to the decomposition $$\mathcal{A} = (VF)_{1} ^{ -1} (Im (I-P_{\alpha_{1}})) \tilde{\oplus} U((VF)_{1} ^{ -1} (R) \tilde{\oplus}  N_{1} ) \stackrel{DF}{\longrightarrow} DV^{-1} (Im (I-P_{\alpha_{1}}))  \tilde{\oplus} (DV^{-1} (R ^{\prime})\tilde{\oplus}  N_{2}^{\prime}   ) =  \mathcal{A} .$$  Since $ U((VF)_{1} ^{ -1} (R) \tilde{\oplus}  N_{1} ) \cong (VF)_{1} ^{ -1} (R) \tilde{\oplus}  N_{1} ,$ by Lemma \ref{r10 l01} we conclude that \\ index $DF=$ index $D+$ index $ F .$ 
	\end{proof} 
Next, in a similar way as in the proof of \cite[Lemma 2.7.13]{MT} we can prove the following lemma. For the convenience of readers, we give the full proof here.
\begin{lemma}\label{invariantlemma1}
	Let $ F \in \mathcal{M} \mathcal{K} \Phi ( \mathcal{A} ) $ and $ K \in \mathcal{F} .$ Then $ F+K \in \mathcal{M} \mathcal{K} \Phi ( \mathcal{A} ) $ and $\text {\rm index } (F+K)= \text {\rm index } F.$
\end{lemma}
	\begin{proof}
		 Let $ F \in \mathcal{M}  \mathcal{K} \Phi (\mathcal{A} ) $ and $$ \mathcal{A} = M_{1} \tilde{\oplus}  N_{1} \stackrel{F}{\longrightarrow} M_{2} \tilde{\oplus} N_{2} = \mathcal{A}  $$  be an $\mathcal{M}  \mathcal{K} \Phi $-decomposition for $F.$ By Proposition \ref{r10 pr01}, there exists some $ \alpha_{0} $ and a closed submodule $ \tilde M $ such that $ Im (I-P_{\alpha_{0}}) \subseteq \tilde M $ and $\mathcal{A} = \tilde M  \tilde{\oplus}  N_{1}.$ Then, by Corollary \ref{extracor7}, $F$ has the matrix 
	$\begin{pmatrix}
		F_{1}  & 0 \\
		0 & F_{4}
	\end{pmatrix} 
$ with respect to the decomposition $$ \mathcal{A} = \tilde M  \tilde{\oplus}  N_{1} \stackrel{F}{\longrightarrow} F(\tilde M ) \tilde{\oplus} N_{2} = \mathcal{A} ,$$ where $F_{1}  $ is an isomorphism. Let $ K \in \mathcal{F} .$ Since $ \lbrace P_{\alpha} \rbrace $ is an approximate unit for $\mathcal{F} ,$ we can find some $ \alpha_{1} \geq \alpha_{0
} $ such that $ \parallel K P_{\alpha_{1}} \parallel \leq \parallel F_{1}^{-1} \parallel ^{-1} .$\\
	We have that $ Im (I-P_{\alpha_{1}}) \subseteq Im (I-P_{\alpha_{0}}) \subseteq \tilde{M}  ,$ hence, by \cite[Lemma 2.6]{IS3} we obtain that $ \tilde{M} = Im (I-P_{\alpha_{1}}) \oplus \mathcal{R} $ where $\mathcal{R} = Im P_{\alpha_{1}} \cap \tilde{M} .$ Since $P_{\mathcal{R}} \leq P_{\alpha_{1}} ,$ we have $ P_{\mathcal{R}} \in \mathcal{F} .$ We get a decomposition 
	$ \mathcal{A} = F_{1} ( Im (I-P_{\alpha_{1}})) \tilde{\oplus} F_{1}( \mathcal{R} ) \tilde{\oplus} N_{2} .$ Since $ F_{1} $ is an isomorphism, by Lemma \ref{r10 l01} we get that $ P_{F_{1}( \mathcal{R} )} \sim P_{ \mathcal{R} } ,$ so $ P_{F_{1}( \mathcal{R} )} \in \mathcal{F} $ as $ P_{ \mathcal{R} } \in \mathcal{F} .$ Moreover, by Lemma \ref{r10 l04} we deduce that $P_{ F_{1} (\mathcal{R}) \tilde{\oplus} N_{2} } \in \mathcal{F} $ and 
	$$[ P_{ F_{1} (\mathcal{R}) \tilde{\oplus} N_{2} } ] = [ P_{ F_{1} (\mathcal{R})} ] + [P_{N_{2}}] = [ P_{ \mathcal{R}} ] + [P_{N_{2}}] .$$ With respect to the decomposition 
	$$ \mathcal{A} = Im (I-P_{\alpha_{1}}) \tilde{\oplus} \mathcal{R} \tilde{\oplus} N_{1} \stackrel{F}{\longrightarrow} 
	F_{1} ( Im (I-P_{\alpha_{1}})) \oplus F_{1} ( Im (I-P_{\alpha_{1}}))^{\perp}  = \mathcal{A} ,$$
	$F$ has the matrix 
	$\begin{pmatrix}
		F_{1}  & \tilde{F_{2}} \\
		0 & \tilde{F_{4}}
	\end{pmatrix} 
	.$
	Let 
	$\begin{pmatrix}
		K_{1}  & K_{2} \\
		K_{3} & K_{4}
	\end{pmatrix} 
	$
	be the matrix of $K$ with respect to the same decomposition. Then 
	$$ \parallel K_{1} \parallel \leq \parallel K_{\vert_{Im \text{ }P_{\alpha_{1}}}} \parallel  = \parallel KP_{\alpha_{1}} \parallel \leq \parallel F_{1}^{-1} \parallel^{-1}.$$
	As in the proof of \cite[Lemma 2.7.13]{MT} we can find unitary operators $\mathcal{U} $ and $ V $ such that 
	$$ \mathcal{A} = Im (I-P_{\alpha_{1}}) \tilde{\oplus} ( \mathcal{U}(\mathcal{R}) \tilde{\oplus} \mathcal{U}(N_{1}) )
	\stackrel{F+K}{\longrightarrow} V( Im (I-P_{\alpha_{1}}) ) \tilde{\oplus} ( F_{1}(\mathcal{R})  \tilde{\oplus} N_{2} ) = \mathcal{A} $$ 
	is an $\mathcal{M} \mathcal{K}\Phi $-decomposition for the operator $F+K .$ Indeed, by Lemma \ref{r10 l04} we have that $ P_{\mathcal{R} \tilde \oplus N_{1}} \in \mathcal{F} $ and $ [P_{\mathcal{R} \tilde \oplus N_{1}}] = [P_{\mathcal{R}}] + [P_{N_{1}}] ,$ whereas by Lemma \ref{r10 l01} we get that $ P_{\mathcal{U} (\mathcal{R} \tilde{\oplus} N_{1} )} \sim P_{\mathcal{R} \tilde \oplus N_{1}} .$ Hence, we deduce that
	$$ index(F+K)=[ P_{\mathcal{U} (\mathcal{R} \tilde{\oplus} N_{1} )}] - [P_{F_{1}(\mathcal{R}) \tilde{\oplus} N_{2} }  ] 
	= [P_{\mathcal{R}}] - [P_{N_{1}}] - [P_{\mathcal{R}}] - [P_{N_{2}}] = index F.$$
\end{proof}
Finally, in a similar way as in the proof of \cite[Theorem 2.7.14]{MT}, we can prove the next theorem. For the convenience of readers, we give the full proof here.
\begin{theorem}
	Let $ F,D, D^{ \prime} \in B(\mathcal{A}). $ If there exist some $ K_{1} , K_{2} \in \mathcal{F}  $ such that  $$ FD = I + K_{1}, \text{ } D^{ \prime} F = I + K_{2}   ,$$ 
	then $ F \in \mathcal{M} \mathcal{K}\Phi (\mathcal{A}) .$
\end{theorem}
	
	\begin{proof}  As in the proof of \cite[Theorem 2.7.14]{MT} we obtain from  Lemma \ref{invariantlemma1} an $\mathcal{M} \mathcal{K} \Phi$-decomposition for $I+K_{1} $ 
	$$\mathcal{A} = M_{1} \tilde{\oplus} N_{1}  \stackrel{I+K_{1}}{\longrightarrow} M_{2} \tilde{\oplus} N_{2} = \mathcal{A}   $$ 
	and we let $\sqcap $ be the projection onto $ N_{2}$ along $M_{2} .$ Then $(I- \sqcap) F $ is an epimorphism onto $M_{2} $ and $D^{\prime} (I- \sqcap) F = I + \tilde{K}_{2} $ for some $\tilde{K}_{2} \in \mathcal{F} .$ This follows since $ \sqcap  \in \mathcal{F}  $ by Remark \ref{skewprojfinite}, so $ D^{ \prime} \sqcap F\in \mathcal{F}  $ because $ \mathcal{F} $ is an ideal. Hence $D^{\prime} (I- \sqcap) F \in \mathcal{M} \mathcal{K}\Phi (\mathcal{A}) $ by Lemma \ref{invariantlemma1}, so there exists an $ \mathcal{M} \mathcal{K} \Phi$-decomposition 
	$$\mathcal{A} = \overline{M}_{1} \tilde{\oplus} \overline{N}_{1}  
	\stackrel{ D^{\prime} ( I-\sqcap ) F}{\longrightarrow} 
	\overline{M}_{2}  \tilde{\oplus} \overline{N}_{2}  = \mathcal{A}   $$ 
	for  $D^{\prime} ( I-\sqcap ) F.$  By the same arguments as in the proof of  \cite[Theorem 2.7.14]{MT} we get that $(I- \sqcap)F_{\vert_{\overline{M}_{1}}} $ is an isomorphism onto $ (I-\sqcap) F ( \overline{M}_{1} )$ and
	$\ker (I-\sqcap)F \subseteq \overline{N}_{1} .$ Since $Im (I-\sqcap)F = M_{2}$ and $(I-\sqcap)F $ is adjointable by  \cite[Corollary 2.5.3]{MT}, by \cite[Theorem 2.3.3]{MT} we have that $ \ker (I-\sqcap)F $ is orthogonally complementable in $\mathcal{A} .$  Hence, by \cite[Lemma 2.6]{IS3} we have that $\ker (I-\sqcap)F $ is orthogonally complementable  in $\overline{N}_{1} .$ Thus, $$ \mathcal{A}= \overline{M}_{1} \tilde{\oplus} \mathcal{R}  \tilde{\oplus} \ker (I-\sqcap)F ,$$ where $\mathcal{R}$ is the orthogonal complement of $\ker (I-\sqcap)F $ in $\overline{N}_{1} .$ However, since $\ker (I-\sqcap)F $ closed and complementable in $\mathcal{A} ,$ by Remark \ref{r10 r01}  $\ker (I-\sqcap)F $ is orthogonally complementable in $\mathcal{A} .$ Hence, since $ \ker (I-\sqcap)F \subseteq \overline{N}_{1} $ and $P_{\overline{N}_{1}} \in \mathcal{F} ,$ we get that $P_{ \ker (I-\sqcap)F } \in \mathcal{F} .$  Since $Im ( I- \sqcap) F ,$ (which is equal to $ M_{2} $), is closed, it follows that $ ( I- \sqcap) F_{\vert_{ ( \overline{M}_{1} \tilde{\oplus} \mathcal{R} ) }} $ is an isomorphism onto $M_{2} = Im ( I- \sqcap) .$ This gives that $F$ has the matrix 
	$\begin{pmatrix}
		F_{1}  & F_{2} \\
		F_{3} & F_{4}
	\end{pmatrix} 
	$ 
	with respect to the decomposition 
	$$\mathcal{A} = ( \overline{M}_{1} \tilde{\oplus} \mathcal{R} )  \tilde{\oplus} \ker ( (I-\sqcap) F) 
	\stackrel{F}{\longrightarrow} M_{2} \tilde{\oplus} N_{2} = \mathcal{A},$$ 
	where $ F_{1} $ is an isomorphism. By the same arguments as in the proof of \cite[Lemma 2.7.10]{MT} we can find isomorphisms $\mathcal{U} $ and $ V $ of $ \mathcal{A} $ such that 
	$$\mathcal{A} = ( \overline{M}_{1} \tilde{\oplus} \mathcal{R} )  \tilde{\oplus} \mathcal{U} (\ker ( (I-\sqcap) F) ) 
	\stackrel{F}{\longrightarrow} V( M_{2} ) \tilde{\oplus} N_{2} = \mathcal{A}$$  
	is a decomposition with respect to which $F$ has the matrix 
	$\begin{pmatrix}
		\tilde{F}_{1}  & 0 \\
		0 & \tilde{F}_{4}
	\end{pmatrix} 
	$ 
	where $\tilde{F_{1}} $ is an isomorphism. Since $P_{\ker (( I- \sqcap) F)} \in \mathcal{F} ,$ by Lemma \ref{r10 l03} we have that $\text{ } \mathcal{U} ({\ker (( I- \sqcap) F)} ) $ is orthogonally complementable in $\mathcal{A} $ and $P_{ \mathcal{U} ({\ker (( I- \sqcap) F)} )} \in \mathcal{F} .$ Thus we have obtained an $\mathcal{M} \mathcal{K} \Phi$-decomposition for the operator $F .$
	\end{proof} 

\begin{remark}
	By Lemma \ref{r10 l02} it follows that our approach to Fredholm theory in unital $ C^{ *} $ -algebras is equivalent to the approach established in \cite{KL}.
\end{remark}

\bibliographystyle{amsplain}

\end{document}